\documentclass[11pt,reqno]{amsart} 

\usepackage{amssymb,latexsym}
\usepackage{mathrsfs}
\usepackage{enumerate}
\usepackage{cite} 

\usepackage[height=190mm,width=130mm]{geometry} 

\theoremstyle{plain}
\newtheorem{theorem}{Theorem}
\newtheorem{lemma}{Lemma}

\theoremstyle{definition}

\theoremstyle{remark}

\numberwithin{equation}{section}

\begin{document}
\nocite{*}
\title[Stability in the class of ...]{Stability in the class of first order delay differential equations}

\author{Eszter Gselmann }
\address{Department on Analysis\\Institute of Mathematics\\University of Debrecen\\ H--4010 Debrecen P.O.Box:12}
\email{gselmann@science.unideb.hu}
\thanks{This research has been supported by the Hungarian Scientific Research Fund
(OTKA) Grant NK 111651 
and by the T\'{A}MOP 4.2.4.A/2-11-1-2012-0001 (Nemzeti Kiv\'{a}l\'{o}s\'{a}g Program --
Hazai hallgat\'{o}i, illetve kutat\'{o}i szem\'{e}lyi t\'{a}mogat\'{a}st biztos\'{i}t\'{o} rendszer kidolgoz\'{a}sa \'{e}s m\H{u}k\"{o}dtet\'{e}se
konvergencia program) project implemented
through the EU and Hungary co-financed by the European Social Fund.}

\author{Anna Kelemen}
\address{Faculty of Informatics\\University of Debrecen\\ H--4010 Debrecen P.O.Box:12}
\email{pandora0303@hotmail.com}

\begin{abstract}
The main aim of this paper is the investigation of the stability problem for 
ordinary delay differential equations. 
More precisely, we would like to study the following problem. 
Assume that for a continuous function a given delay differential equation is 
fulfilled \emph{only approximately}. Is it true that in this case this function 
is \emph{close} to an exact solution of this delay differential equation?
\end{abstract}



\maketitle

\section{Introduction and preliminaries}

Whenever we would like to analyze mathematically a problem, arising in the 'real' world, first we have to 
select a mathematical model for formulation of our problem. 
This means that we have to choose or invent some mathematical problem to present our 'real world problem'. 
The mathematical models chosen are often differential equations. 
Mostly this differential equations are rather special, more precisely, 
in them the unknown functions and its derivatives are all evaluated at \emph{the same 
instant} $t$. 

A more general type of differential equations, called \emph{functional differential equations}, are those in which 
the unknown functions occur with different arguments. Such type of equations are rather natural for example in 
control systems, since any system involving feedback control will certainly involve \emph{time delays}. 
These arise because a finite time is required to sense information and then react to it. 

We remark an other problem which comes from electrodynamics, the \emph{two-body problem}. Let us consider two electrons and try to 
describe the interaction between them. We have to take it into account that the interactions between the two particles travel not instantaneously, 
but with finite speed $c$ (the speed of light). 
Let $x_{1}(t)$ and $x_{2}(t)$ represent the positions of the two particles at time $t$. The the fields reaching $x_{1}(t)$ at time $t$ from the second particle 
were generated at $t-r$, where the delay must satisfy 
\[
 cr=\left| x_{1}(t)-x_{2}(t-r)\right|. 
\]
Clearly, the (unknown) delay $r$ is not a constant but it depends on $t$. 
In terms of these delays one can obtain the following system for the motion 
\[
 \begin{array}{rcl}
  x_{1}''(t)&=&f_{1}\left(x_{1}(t)-x_{2}(t-r_{21}(t)), x_{1}'(t), x_{2}'(t-r_{21}(t))\right)\\
   x_{2}''(t)&=&f_{2}\left(x_{2}(t)-x_{1}(t-r_{12}(t)), x_{2}'(t), x_{1}'(t-r_{12}(t))\right),  
 \end{array}\]
with certain maps $f_{1}, f_{2}\colon \mathbb{R}^{3}\to \mathbb{R}$.

In a real world applications, initial conditions are almost never known exactly. Rather,
experimental and physical errors will only allow us to say that their values are 
approximately equal to those in our mathematical model. Thus, to retain physical relevance, we
need to be sure that small errors in our initial measurements
do not induce a large change
in the solution. A similar argument can be made for any physical parameters, e.g., masses,
charges, spring stiffnesses, frictional coefficients, etc.,
that appear in the differential equation itself. 
A slight change in the parameters should not have a dramatic effect on the solution. 
Such type of results are referred to \emph{continuous dependence on the initial values}. 

It can also however happen that not only the initial value but also 
the equation is satisfied 'approximately'. These type of investigations began in the 1990s. 
For example in \cite{AlsinaGer} C.~Alsina and R.~Ger studied the following problem. 
Let $\varepsilon>0$ be given, $I\subset \mathbb{R}$ be a nonvoid open interval and 
assume that for the differentiable function $x\colon I\to \mathbb{R}$ inequality 
\[
 \left|x'(t)-x(t)\right| <\varepsilon 
\qquad 
\left(t\in I\right). 
\]
Is it true that there exist  constants $C, \kappa$ such that 
\[
 \left|x(t)-C\exp(t)\right|<\kappa\varepsilon
\]
is fulfilled for any $x\in I$. In the above mentioned paper (among others) 
this problem is answered in the affirmative. 
Roughly speaking this result expresses the following. Assume that for the differentiable function 
$x\colon I\to \mathbb{R}$ differential equation 
\[
 x'(t)=x(t) 
\qquad 
\left(t\in I\right)
\]
is fulfilled \emph{only approximately}. Is it true that in such a situation there exists a solution to this 
differential equation $\widetilde{x}\colon I\to \mathbb{R}$ such that the functions 
$x$ and $\widetilde{x}$ are \emph{close to each other}?

Later, in \cite{TakMiuMiy02} the first order equation 
\[
 x'(t)=\lambda x(t)
\]
was considered for Banach space valued functions, and the same question was answered affirmatively. 

Recently the study of such problems became an intensively developing  area, see for instance 
Jung \cite{Jung10}, \cite{Jung12} and Andr\'{a}s--Kolumb\'{a}n \cite{AndrasKolumban}. 
For example, in the latter work using the so-called Perov fixed point theorem, the authors 
studied the following system 
\[
 \begin{array}{rcl}
  x'(t)&=&f_{1}(t, x(t), y(t))\\
y'(t)&=&f_{2}(t, x(t), y(t)) \qquad \text{almost everywhere}\\
x(0)&=&\alpha[x]\\
y(0)&=&\beta[y]
 \end{array}
\]
on bounded as well as on unbounded subset of $]0, +\infty[$. Nevertheless, the use of Perov fixed point theorem, allowed only that 
the functions $f_{1}, f_{2}$ are Lipschitz functions with \emph{small} Lipschitz constant. 

Furthermore, in \cite{JungBrzdek} the stability
problem for a rather special delay differential equations was investigated, namely
\[x (t) = \lambda x(t -\tau )\]
As we will see later, the stability problem of this equation will be a particular case
of ours with the choice $m = 1, g1 (t) = t -\tau$ and
\[f \left(t, x (g_{1}(t))\right) = \lambda x (g_{1} (t)) = \lambda x(t -\tau ).\]
For further results concerning stability of functional as well as ordinary differ-
ential equations we refer also to \cite{Jung04, FechGer}.

Therefore the main aim of this paper is to investigate and to extend the above presented problem for a rather large class 
of first order ordinary delay differential equations. 

First of all, we will list the notation and the terminology that will be used subsequently. 
While doing so we will relay on the two basic monographs Driver \cite{Dri77} (concerning delay differential equations) and 
Walter \cite{Wal64} (concerning integral inequalities).

Let $J=[t_{0}, \beta[\subset \mathbb{R}$ be an interval and 
$D\subset \mathbb{R}^{n}$ be an open set. 
Let further $f\colon J\times D^{m}\to \mathbb{R}^{n}$ be a mapping and 
$g_{1}, \ldots, g_{m}\colon J\to \mathbb{R}$ be functions so that 
\[
 \gamma \leq g_{j}(t)\leq t 
\qquad 
\left(t\in [t_{0}, \beta[\right)
\]
for all $j=1, \ldots, m$ with a certain real number $\gamma$. Finally, let 
$\theta\colon [\gamma, t_{0}]\to D$ be a given initial function. 
Consider the system of delay differential equations 
\[
 x'(t)= f\left(t, x(g_{1}(t)), \ldots, x(g_{m}(t))\right). 
\]
Given $t\in [t_{0}, \beta[$ and any function 
$\chi\colon [\gamma, t]\to D$ define 
\[
 F(t, \chi_{t})= 
f\left(t, x(g_{1}(t)), \ldots, x(g_{m}(t))\right). 
\]
With this notation, our main aim is to investigate the 
Hyers--Ulam stability of the following problem. 
\[
 \tag{$\mathscr{P}$}\label{prob} 
\begin{array}{rcl}
 x'(t)&=&F(t, x_{t})\qquad t\geq t_{0}\\
x(t)&=&\theta(t) \qquad t\in [\gamma, t_{0}]. 
\end{array}
\]
A solution of this problem \eqref{prob} is a continuous function 
$x\colon [\gamma, \beta_{1}[\to D$ for some $\beta_{1}\in [\gamma, t_{0}]$ such that 
\begin{enumerate}[(i)]
 \item for all $t\in [\gamma, \beta_{1}]$ we have $(t, x_{t})\in [t_{0}, \beta_{1}[\times D^{m}$
\item for any $t\in [t_{0}, \beta_{1}[$ 
\[
 x'(t)=F(t, x_{t})
\]
is fulfilled. 
\item $x(t)=\theta(t)$ for all $t\in [\gamma, t_{0}]$. 
\end{enumerate}

Henceforth we will assume that $f$ and $g_{1}, \ldots, g_{m}$ are continuous. Thus if 
$x$ is continuous, then the mapping 
\[
[\gamma, \beta_{1}[ \ni t \longmapsto F(t, x_{t})
\]
is continuous, as well. In such a situation however 
\[
 \tag{$\mathscr{I}$}\label{I} 
x(t)=
\begin{cases}
 \theta(t), & \text{ if } t\in [\gamma, t_{0}]\\
\theta(t_{0})+\displaystyle\int_{t_{0}}^{t}F(s, x_{s})ds, & \text{ if } t\in [t_{0}, \beta_{1}[. 
\end{cases}
\]

The following lemma will be utilized during our main results, 
see also Walter \cite[Theorem III. (Gronwall’s inequality), page 14]{Wal64}. 

\begin{lemma}\label{L1.1}
 Let $J\subset \mathbb{R}$ be an interval, 
$g, v\colon J\to \mathbb{R}$ be continuous functions and 
$h\colon J\to \mathbb{R}$ be a nonnegative, Lebesgue integrable function such that 
\[
 v(t)\leq g(t)+\int_{t_{0}}^{t}h(s)v(s)ds 
\qquad 
\left(t\in J\right). 
\]
Then 
\[
 v(t)\leq g(t)+\int_{t_{0}}^{t}g(s)h(s)e^{H(t)-H(s)}ds 
\qquad 
\left(t\in J\right). 
\]
Furthermore, if $g$ is absolutely continuous, then 
\[
 v(t)\leq e^{H(t)}\left(g(t_{0})+\int_{t_{0}}^{t}g'(s)e^{H(s)}ds\right)
\]
holds, where 
\[
 H(t)=\int_{t_{0}}^{t}h(s)ds 
\qquad 
\left(t\in J\right). 
\]
\end{lemma}

Finally we remark the without the loss of generality it is enough to restrict ourselves to 
first order delay differential equations. 
Indeed, 
let $J=[t_{0}, \beta[$ be an interval, 
$D\subset \mathbb{R}^{n}$ be an open set. 
Let $F\colon J\times D^{m}\to \mathbb{R}$ be a continuous function.  
Let further $g_{1}, \ldots, g_{m}\colon J\to \mathbb{R}$ be continuous functions so that 
\[
 \gamma\leq g_{j}(t)\leq t 
\qquad 
\left(t\in [t_{0}, \beta[\right)
\]
for all $j=1, \ldots, m$ with a certain real number $\gamma$. 
Finally, let $\theta \colon [\gamma, t_{0}]\to \mathbb{R}$ be $(n-1)$ times continuously differentiable initial function.
Assume that the $(n-1)$ times continuously differentiable function $x\colon [\gamma, \beta[\to \mathbb{R}$ is a solution of the problem 
\[
 \begin{array}{rcl}
  x^{(n)}(t)&=&F\left(t, x(g_{1}(t)), \ldots, x^{(n-1)}(g_{m}(t))\right)\quad \left(t\in [t_{0}, \beta[\right)\\
x(t)&=&\theta(t) \qquad \left(t\in [\gamma, t_{0}]\right)
 \end{array}
\]
Define the function $y\colon [\gamma, \beta[\to \mathbb{R}^{n}$ by 
\[
 y(t)=\begin{pmatrix}
       x(t)\\
x'(t)\\
\vdots\\
x^{(n-1)(t)}
      \end{pmatrix}
\qquad 
\left(t\in [\gamma, \beta[\right). 
\]
Then the continuous function $y$ is a solution of the following first order delay system
\[
 \begin{array}{rcl}
  y'_{1}(t)&=&y_{2}(t)\\
y'_{2}(t)&=&y_{3}(t)\\
&\vdots& \\
y'_{n-1}(t)&=&y_{n}(t)\\
y'_{n}(t)&=& F(t, y(g_{1}(t)), \ldots, y(g_{m}(t))
 \end{array}
\qquad 
\left(t\in [t_{0}, \beta[\right), 
\]
furthermore 
\[
 y(t)=\begin{pmatrix}
       \theta(t)\\
\theta'(t)\\
\vdots\\
\theta^{(n-1)}(t)
      \end{pmatrix}
\qquad 
\left(t\in [\gamma, t_{0}]\right), 
\]
where $y_{i}\colon [\gamma, \beta[\to \mathbb{R}$ denotes the $i$th coordinate function of the 
function $y$ for all $i=1, \ldots, n$. 
Furthermore, the converse implication is also valid. Namely, if the function $y\colon [\gamma, \beta[\to \mathbb{R}^{n}$ is a 
solution of the latter system, then the first coordinate function $y_{1}\colon [\gamma, \beta[\to \mathbb{R}$ is a solution of the above 
$n$th order delay equation.

\section{Main results}

In this section, our main results will be divided into four parts. 
First we will prove a basic lemma that will be utilized later. 
After that (with the aid of Lemma \ref{L2.1}) we will present a 
\emph{uniqueness theorem}. Afterwards a \emph{continuous dependence on the initial function} type result will 
follow. 
Finally, we will end our paper with a stability type result.

\subsection{A useful lemma}

\begin{lemma}\label{L2.1}
Let $J=[t_{0}, \beta[\subset \mathbb{R}$ be an interval and 
$D\subset \mathbb{R}^{n}$ be an open set. 
Let further $f, \widetilde{f}\colon J\times D^{m}\to \mathbb{R}^{n}$ be mappings and 
$g_{1}, \ldots, g_{m}\colon J\to \mathbb{R}$ be functions so that 
\[
 \gamma \leq g_{j}(t)\leq t 
\qquad 
\left(t\in [t_{0}, \beta[\right)
\]
for all $j=1, \ldots, m$ with a certain real number $\gamma$. Finally, let 
$\theta, \widetilde{\theta}\colon [\gamma, t_{0}]\to D$ be given initial functions. 
Assume that the functions $x, \widetilde{x}\colon [\gamma, \beta[\to D$ are solutions to problems 
\[
\begin{array}{rcl}
 x'(t)&=&F(t, x_{t})\qquad t\geq t_{0}\\
x(t)&=&\theta(t) \qquad t\in [\gamma, t_{0}]. 
\end{array} 
\]
and 
\[
 \begin{array}{rcl}
 \widetilde{x}'(t)&=&\widetilde{F}(t, \widetilde{x_{t}})\qquad t\geq t_{0}\\
\widetilde{x}(t)&=&\widetilde{\theta}(t) \qquad t\in [\gamma, t_{0}]. 
\end{array}
\]
respectively. 
Suppose further that we also have the following. 
\begin{enumerate}[(i)]
 \item there exists a continuous nonnegative function 
$h\colon J\to \mathbb{R}$ such that 
for all $t\in [\gamma, \beta[$ and for any $z\in D^{m}$
\[
 \left\|F(t, z)-\widetilde{F}(t, z) \right\| \leq h(t)\left\| z\right\|. 
\]
\item there exists a continuous nonnegative function 
$k\colon J\to \mathbb{R}$ such that 
\[
\left\|F(t, z)-F(t, z') \right\|\leq k(t)\left\| z-z'\right\|
\]
for all $t\in [\gamma, \beta[$ and for any $z, z'\in D^{m}$. 
\end{enumerate}
Then 
\[
 \left\|x(t)-\widetilde{x}(t) \right\|
\leq 
\left\|\theta-\widetilde{\theta} \right\|
\exp\left(\int_{t_{0}}^{t}h(s)+k(s)ds\right)
\qquad 
\left(t\in J\right). 
\]
\end{lemma}
\begin{proof}
 Applying integral inequality \ref{I}, we immediately get that 
\[
x(t)=\theta(t_{0})+\displaystyle\int_{t_{0}}^{t}F(s, x_{s})ds
\]
and 
\[
\widetilde{x}(t)= \widetilde{\theta}(t_{0})+\displaystyle\int_{t_{0}}^{t}F(s, \widetilde{x}_{s})ds
\]
holds for all $t\in J$. 
Therefore 
\begin{multline*}
 \left\| x(t)-\widetilde{x}(t)\right\|
\\
=
\left\| \theta(t_{0})+\displaystyle\int_{t_{0}}^{t}F(s, x_{s})ds
- \left(\widetilde{\theta}(t_{0})+\displaystyle\int_{t_{0}}^{t}F(s, \widetilde{x}_{s})ds\right)
\right\| 
\\
\leq 
\left\| \theta(t_{0})-\widetilde{\theta}(t_{0})\right\| 
+
\left\|  \int_{t_{0}}^{t} F(s, x_{s})-\widetilde{F}(s, \widetilde{s}_{s})ds \right\|
\\
\leq 
\left\| \theta(t_{0})-\widetilde{\theta}(t_{0})\right\|
+
\left\|  \int_{t_{0}}^{t} F(s, x_{s})-F(s, \widetilde{x}_{s})ds \right\|
+
\left\|  \int_{t_{0}}^{t} F(s, \widetilde{x}_{s})-\widetilde{F}(s, \widetilde{x}_{s})ds \right\|
\\
\leq 
\left\| \theta(t_{0})-\widetilde{\theta}(t_{0})\right\|
+\int_{t_{0}}^{t}k(s)\sup_{\gamma\leq \xi\leq s}\left\|x(\xi)-\widetilde{x}(\xi) \right\|ds
\\+
\int_{t_{0}}^{t}h(s)\sup_{\gamma\leq \xi\leq s}\left\|x(\xi)-\widetilde{x}(\xi) \right\|ds
\\
\leq 
\left\| \theta(t_{0})-\widetilde{\theta}(t_{0})\right\|
+\int_{t_{0}}^{t}(k(s)+h(s))\sup_{\gamma\leq \xi\leq s}\left\|x(\xi)-\widetilde{x}(\xi) \right\|ds
\end{multline*}
Define the function $v\colon [\gamma, \beta[\to \mathbb{R}$ by 
\[
 v(t)=\sup_{\gamma\leq \xi\leq t}\left\|x(\xi)-\widetilde{x}(\xi) \right\| 
\qquad 
\left(t\in [\gamma, \beta[\right). 
\]
Then 
\[
 \left\|x(t)-\widetilde{x}(t) \right\|
\leq 
v(t_{0})+\int_{t_{0}}^{t}(k(s)+h(s))v(s)ds
\qquad 
\left(t\in J\right), 
\]
or, since 
\[
 \left\|\theta(t)-\widetilde{\theta}(t) \right\|\leq v(t_{0})
\qquad 
\left(t\in [\gamma, t_{0}]\right)
\]
holds, we have 
\[
 v(t)\leq v(t_{0})+\int_{t_{0}}^{t}(k(s)+h(s))v(s)ds
\]
for all $t\in J$. Now, Lemma \ref{L1.1} implies that 
\[
 \left\|x(t) -\widetilde{x}(t)\right\|
\leq \left\|\theta -\widetilde{\theta} \right\| 
\exp \left(\int_{t_{0}}^{t}k(s)+h(s)ds\right)
\]
for all $t\in J$, which ends the proof. 
\end{proof}

\subsection{Corollaries}

\subsubsection*{A uniqueness theorem}

Based on the previous results, now we are able to prove the following uniqueness theorem. 

\begin{theorem}
 Let $J=[t_{0}, \beta[$ be an interval, 
$D\subset \mathbb{R}^{n}$ be an open set. 
Let $F\colon J\times D^{m}\to \mathbb{R}^{n}$ be a function so that 
\[
 \left\|F(t, z) -F(t, z')\right\|\leq k(t)\left\|z-z' \right\| 
\qquad 
\left(t\in J, z\in D^{m}\right)
\]
is fulfilled with a certain continuous function $k\colon J\to \mathbb{R}$. 
Let further $g_{1}, \ldots, g_{m}\allowbreak\colon  J\to \mathbb{R}$ be continuous functions so that 
\[
 \gamma\leq g_{j}(t)\leq t 
\qquad 
\left(t\in [t_{0}, \beta[\right)
\]
for all $j=1, \ldots, m$ with a certain real number $\gamma$. 
Finally, let $\theta\colon [\gamma, t_{0}]\to \mathbb{R}$ be a given continuous initial function. 
Then problem \eqref{prob} has at most one solution on any interval $[\gamma, \beta_{1}[$, where 
$t_{0}< \beta_{1}\leq \beta$. 
\end{theorem}
\begin{proof}
 Suppose on the contrary that there are two different 
solutions, say $x$ and $\widetilde{x}$. 
Since $x(t)=\widetilde{x}(t)$ on $[\gamma, t_{0}]$, there exists $t\in ]t_{0}, \beta_{1}[$ such that 
$x(t)\neq \widetilde{x}(t)$. 
Let 
\[
 t_{1}=\inf\left\{t\in ]t_{0}, \beta_{1}[\, \vert \, x(t)\neq \widetilde{x}(t)\right\}. 
\]
Then $t_{1}\in ]t_{0}, \beta_{1}[$ and due to the continuity of $x$ and $\widetilde{x}$ we have 
\[
 x(t)=\widetilde{x}(t)
\]
for all $t\in [\gamma, t_{1}[$. 
Since $t_{1}\in ]t_{0}, \beta_{1}[$ and for all $j=1, \ldots, m$ 
$x(g_{j}(t_{1}))$ belongs to the open set $D$, there are positive real numbers 
$\rho_{1}, \rho_{2}$ such that 
\[
 [t_{1}-\rho, t_{1}+\rho]\subset ]t_{0}, \beta_{1}[ 
\quad 
\text{and}
\quad 
A_{j}=
\left\{\xi\in \mathbb{R}^{n}\, \vert \, \left\|x(g_{j}(t))-\xi\right\|\leq \rho_{2} \right\}\subset D. 
\]
Let 
\[
 K=\sup_{t\in [t_{1}-\rho_{1}, t_{1}+\rho_{1}]}k(t). 
\]
Since $k$ is continuous on the compact interval $[t_{1}-\rho_{1}, t_{1}+\rho_{1}]$ we have $K\in \mathbb{R}$. 

Again, due to the continuity of $x$ and $g_{1}, \ldots, g_{m}$, there exists $\beta_{2}\in ]t_{1}, t_{1}+\rho_{1}[$ 
such that 
$x(g_{j}(t)), \widetilde{x}(g_{j}(t))\in A_{j}$ for all $j=1, \ldots, m$ and 
$t\in ]t_{1}, \beta_{2}[$. 
Let now $t\in [t_{1}, \beta_{2}[$, then 
\[
 \left\|x(t) -\widetilde{x}(t)\right\|
=
\left\|\int_{t_{1}}^{t}F(s, x_{s})-F(s, \widetilde{x}_{s})ds \right\|
\leq K\int_{t_{1}}^{t}\sup_{\gamma\leq \xi\leq s}\left\| x(\xi)-\widetilde{x}(\xi)\right\| ds. 
\]
Define 
\[
 v(t)=\sup_{\gamma\leq \xi\leq t} \left\| x(\xi)-\widetilde{x}(\xi)\right\| 
\qquad 
\left(t\in [\gamma, \beta_{1}[\right). 
\]
Then 
\[
 \left\| x(t)-\widetilde{x}(t)\right\| \leq K\int_{t_{1}}^{t}v(s)ds
\qquad 
\left(t\in [t_{1}, \beta_{2}[\right), 
\]
or, since $x\equiv \widetilde{x}$ on $[\gamma, t_{1}]$ we have 
\[
 v(t)\leq K\int_{t_{1}}^{t}v(s)ds 
\qquad 
\left(t\in [t_{1}, \beta_{2}[\right). 
\]
In view of Lemma \ref{L1.1}., this implies that $v(t)=0$ for all $t\in [t_{1}, \beta_{2}[$, that is, 
\[
 x(t)=\widetilde{x}(t) 
\qquad 
\left(t\in [t_{1}, \beta_{1}[\right), 
\]
which contradicts to the definition of $t_{1}$. 
\end{proof}

We remark that the case when 
\[
 k(t)=k 
\qquad 
\left(t\in J\right)
\]
can be found in the monograph of Driver, see \cite[Theorem A., p. 259]{Dri77}. 

\subsubsection*{Continuous dependence on the initial function}

Using Lemma \ref{L2.1}. with the choice 
$h \equiv 0$, the following result can be derived immediately.

\begin{theorem}
 Let $J=[t_{0}, \beta[$ be an interval, 
$D\subset \mathbb{R}^{n}$ be an open set. 
Let $F\colon J\times D^{m}\to \mathbb{R}^{n}$ be a function so that 
\[
 \left\|F(t, z) -F(t, z')\right\|\leq k(t)\left\|z-z' \right\| 
\qquad 
\left(t\in J, z\in D^{m}\right)
\]
is fulfilled with a certain continuous function $k\colon J\to \mathbb{R}$. 
Let further $g_{1}, \ldots, g_{m}\allowbreak\colon J\to \mathbb{R}$ be continuous functions so that 
\[
 \gamma\leq g_{j}(t)\leq t 
\qquad 
\left(t\in [t_{0}, \beta[\right)
\]
for all $j=1, \ldots, m$ with a certain real number $\gamma$. 
Finally, let $\theta, \widetilde{\theta}\colon [\gamma, t_{0}]\to \mathbb{R}$ be a given continuous initial functions. 
If the functions $x, \widetilde{x}\colon [\gamma, \beta[\to \mathbb{R}$ are solutions to the problems 
\[
 \begin{cases}
  x'(t)=F(t, x_{t})& \text{ for $t\geq t_{0}$}\\
x(t)=\theta(t)& t\in [\gamma, t_{0}]
 \end{cases}
\quad 
\text{and}
\quad 
\begin{cases}
  \widetilde{x}'(t)=F(t, \widetilde{x}_{t})& \text{ for $t\geq t_{0}$}\\
\widetilde{x}(t)=\widetilde{\theta}(t)& t\in [\gamma, t_{0}]
 \end{cases}. 
\]
Then for all $t\in [t_{0}, \beta[$
\[
 \left\| x(t)-\widetilde{x}(t)\right\| 
\leq \left\| \theta-\widetilde{\theta}\right\| \exp\left(\int_{t_{0}}^{t}k(s)ds\right). 
\]
\end{theorem}

\subsubsection*{A stability result}

Finally, we end this paper with the following stability type result. 

\begin{theorem}
 Let $\varepsilon>0$ be arbitrarily fixed, $J=[t_{0}, \beta[$ be an interval and 
$D\subset \mathbb{R}^{n}$ be an open set. 
Let $F\colon J\times D^{m}\to \mathbb{R}^{n}$ be a function so that 
\[
 \left\|F(t, z) -F(t, z')\right\|\leq k(t)\left\|z-z' \right\| 
\qquad 
\left(t\in J, z\in D^{m}\right)
\]
is fulfilled with a certain continuous function $k\colon J\to \mathbb{R}$. 
Let further $g_{1}, \ldots, g_{m}\allowbreak\colon J\to \mathbb{R}$ be continuous functions so that 
\[
 \gamma\leq g_{j}(t)\leq t 
\qquad 
\left(t\in [t_{0}, \beta[\right)
\]
for all $j=1, \ldots, m$ with a certain real number $\gamma$. 
Finally, let $\theta\colon [\gamma, t_{0}]\to \mathbb{R}$ be a given continuous initial function. 
Assume that the continuous initial function 
$x\colon [t_{0}\beta[\to \mathbb{R}$ fulfills 
\[
 \begin{array}{rcl}
  \left\|x'(t)-F(t, x_{t}) \right\|&\leq &\varepsilon \qquad \left(t\in [t_{0}, \beta[\right)\\
x(t)&=&\theta(t) \qquad \left(t\in [\gamma, t_{0}]\right). 
 \end{array}
\]
Then there exists a uniquely determined solution $\widetilde{x}\colon [t_{0}, \beta[\to \mathbb{R}$ of the problem 
\[
 \begin{array}{rcl}
  x'(t)&=&F(t, x_{t}) \qquad \left(t\in [t_{0}, \beta[\right)\\
x(t)&=&\theta(t) \qquad \left(t\in [\gamma, t_{0}]\right). 
 \end{array}
\]
such that 
\[
 \left\|x(t)-\widetilde{x}(t) \right\|\leq \varepsilon(t-t_{0})\exp\left(\int_{t_{0}}^{t}k(s)ds\right)
\]
for all $t\in [t_{0}, \beta[$. 
\end{theorem}
\begin{proof}
 Assume that for all $t\in [t_{0}, \beta[$ we have 
\[
 \left\| x'(t)-F(t, x_{t})\right\| \leq \varepsilon. 
\]
Equivalently, this means that $x$ is a solution to the problem 
\[
 \begin{array}{rcl}
  x'(t)&=&\widetilde{F}(t, x_{t}) \qquad \left(t\in [t_{0}, \beta[\right)\\
x(t)&=&\theta(t) \qquad \left(t\in [\gamma, t_{0}]\right), 
 \end{array}
\]
where 
\[
 \widetilde{F}(t, z)=F(t, z)+b(t) 
\qquad 
\left(t\in [t_{0}, \beta[, z\in D^{m}\right), 
\]
with a certain bounded function $b\colon J\to \mathbb{R}$. 

Let $\widetilde{x}\colon J\to \mathbb{R}$ be the uniquely determined solution of the problem 
\[
 \begin{array}{rcl}
  \widetilde{x}'(t)&=&F(t, \widetilde{x}_{t}) \qquad \left(t\in [t_{0}, \beta[\right)\\
\widetilde{x}(t)&=&\theta(t) \qquad \left(t\in [\gamma, t_{0}]\right). 
 \end{array}
\]
In this case however 
\begin{multline*}
 \left\| x(t)-\widetilde{x}(t)\right\|
=
\left\|\int_{t_{0}}^{t}\widetilde{F}(s, x_{s})-F(s, \widetilde{x}_{s})ds \right\|
\\
\leq 
\int_{t_{0}}^{t}k(s)\sup_{\gamma\leq \xi\leq s}\left\|x(s)-\widetilde{x}(s) \right\|ds +\varepsilon (t-t_{0})
\end{multline*}
for any $t\in [t_{0}, \beta[$. 
Similarly as above, define the function 
$v\colon [\gamma, \beta[\to \mathbb{R}$ through 
\[
 v(t)=\sup_{\gamma\leq \xi\leq t}\left\|x(t)-\widetilde{x}(t) \right\| 
\qquad 
\left(t\in [\gamma, \beta[\right). 
\]
Then the previous inequality simply yield that 
\[
 v(t)\leq \varepsilon(t-t_{0})+\int_{t_{0}}^{t}k(s)v(s)ds 
\qquad 
\left(t\in [\gamma, \beta[\right).
\]
Thus, in virtue of Lemma \ref{L1.1}, for all $t\in [t_{0}, \beta[$
\[
 \left\|x(t)-\widetilde{x}(t) \right\|\leq \varepsilon(t-t_{0})\exp\left(\int_{t_{0}}^{t}k(s)ds\right). 
\]
\end{proof}

\subsubsection*{Acknowledgement}

This paper is dedicated to Professor \'{A}rp\'{a}d Sz\'{a}z (University of Debrecen) on the occasion of this $70$\textsuperscript{th} birthday.


\end{document}